\documentclass[10pt]{article}
\usepackage{amsmath}
\usepackage{amssymb}
\usepackage{amsfonts}
\usepackage{amscd}
\usepackage[usenames]{color}
\usepackage{amsthm}

\usepackage[top=1.22in, bottom=1.47in, left=1.68in, right=1.68in]{geometry} 

\usepackage{layout}

\usepackage[hidelinks,draft=false]{hyperref}
\usepackage{bookmark}

\bookmarksetup{
  open,
  numbered,
  addtohook={%
    \ifnum\bookmarkget{level}>2 
      \renewcommand*{\numberline}[1]{}%
    \fi
  },
}

\usepackage{fancybox}

\usepackage{chngcntr}
\counterwithin*{equation}{section}

\usepackage{bm}

\usepackage{mathtools} 

\usepackage{stmaryrd} 

\input{xypic}
\xyoption{all}

\newtheorem{theorem}{Theorem}[section]

\newtheorem{lemma}[theorem]{Lemma}

\newtheorem{defn}[theorem]{Definition}

\newtheorem{remark}[theorem]{Remark}

\newtheorem{example}[theorem]{Example}

\newtheorem{exercise}[theorem]{Exercise}

\newtheorem{terminology}[theorem]{Terminology}

\newtheorem{notation}[theorem]{Notation}

\newtheorem{observation}[theorem]{Observation}

\newtheorem{question}[theorem]{Question}
\newenvironment{q}{\begin{question}}{\end{question}}

\def\la{\lambda}

\def\a{\alpha}
\def\vs{\varsigma}

\def\b{\beta}

\def\o{\omega}
\def\vp{\varphi}
\def\s{\sigma}
\def\ep{\epsilon}

\def\L{\Lambda}
\def\G{\Gamma}

\def\O{\mathcal O}

\def\Z{\mathbb Z}
\def\C{\mathbb C}
\def\Q{\mathbb Q}

\def\P{\mathbb P}
\def\w{\wedge}

\def\hphi{\hat{\phi}}

\def\hE{\hat{E}}

\def\hF{\hat{F}}

\def\vac{v_\emptyset}
\def\hb{\hat{\b}}

\def\p{\text{\bf p}}

\def\OP{\text{OP}}

\def\SP{\text{SP}}
\def\PP{\text{P}}

\def\1{{\bm 1}}
\def\f{\text{\bf f}}
\def\0{{\bm 0}}

\def\vac{{v_\emptyset}}

\def\bu{\bullet}

\def\dim{{\rm dim}}

\def\non{\noindent}
\def\Aut{{\rm Aut}}

\def\u{\underline}

\def\lb{\langle}
\def\rb{\rangle}

\def\ds{\displaystyle}

\def\Cliff{{\textsf{Cliff}}}
\def\SS{{\textsf{S}}}
\def\SC{{\textsf{C}}}

\def\f{{\bf f}}

\def\sb{{\scriptscriptstyle{B}}}

\def\reg{\text{reg}}

\title{{\bf A square root of Hurwitz numbers}}

\author{Junho Lee}

\date{\empty}
\addtocounter{section}{0}
\begin{document}

\maketitle

\begin{abstract}

We exhibit a generating function of  spin Hurwitz numbers analogous to  (disconnected) double Hurwitz numbers that is a tau function of the two-component BKP (2-BKP)  hierarchy and is a square root of a tau function of the two-component KP (2-KP)  hierarchy defined by related Hurwitz numbers.

\end{abstract}



\section{Introduction}

The Gromov-Witten theory of K\"{a}hler surfaces (with smooth canonical divisor) is entirely determined by the GW theory of spin curves (see \cite{LP1, KiL, MP}).
The (dimension zero) sum formula for spin curves (Theorem~1.1 of \cite{LP2}) indicates that
the spin Hurwitz theory is to the GW theory of spin curves what the Hurwitz theory is to the GW theory of curves.
For this reason, it would be interesting to find some connections between the following theories:
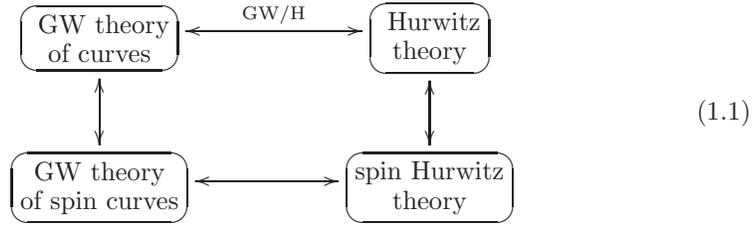
\begin{gather}
\begin{aligned}
\xymatrix{
\ovalbox{\txt{\ GW theory\  \\ of curves}} \ar@{<->}[rr]^-{\text{GW/H}} \ar@{<->}[d]_{}
&& \ovalbox{\txt{ Hurwitz \\ theory }} \ar@{<->}[d]^-{\text{}}
\\
\ovalbox{\txt{GW theory  \\ of spin curves}} \ar@{<->}[rr]^-{}
&& \ovalbox{\txt{spin Hurwitz \\ theory }}
}
\end{aligned}
\label{figure}
\end{gather}
In the top arrow is the celebrated GW/H correspondence  developed in \cite{OP2,OP3}. This paper aims to find a connection in the right arrow  and also to address a question on a correspondence in the bottom arrow analogous to the GW/H correspondence.

In \cite{O}, A. Okounkov showed that a generating function of (disconnected) double Hurwitz numbers is a tau function of the 2-Toda lattice hierarchy.
The key idea is to write the generating function in terms of  Schur functions that are special tau functions of the KP hierarchy.

Following the same idea, we will compare the generating functions of Hurwitz numbers and spin Hurwitz numbers defined below.  The Hurwitz generating function can be written via  Schur functions and the spin Hurwitz generating function via Schur Q-functions that are tau functions of the BKP hierarchy.

Given partitions $\mu,\nu,\eta_2=(2,1^{d-2})$ and $\eta_3=(3,1^{d-3})$ of $d>0$ and integers $r_2,r_3\geq 0$, the Hurwitz number of $\P^1$
is the weighted sum
\begin{equation}\label{Hn}
H^0_d(\mu,\nu,\eta_2^{r_2},\eta_3^{r_3}) =\sum_f \frac{1}{|\Aut(f)|}
\end{equation}
of possibly disconnected ramified covers $f:C\to\P^1$ with ramification profiles $\mu$ over $0\in\P^1$, $\nu$ over $\infty\in\P^1$ and  $\eta_2$ and $\eta_3$ over $r_2$ and $r_3$ other fixed points of $\P^1$.
Here $\eta_2$ denotes the simple ramification and the automorphism group $\Aut(f)$ of $f$ consists of automorphisms $\vs$ of the domain curve $C$ satisfying $f\circ \vs=f$.
The domain Euler characteristic $\chi(C)$ is related to the partition lengths $\ell(\mu)$ and $\ell(\nu)$ and the numbers $r_2$ and $r_3$ by the Riemann-Hurwitz formula
$$
\chi(C)=\ell(\mu)+\ell(\nu)-r_2-2r_3.
$$
In particular, this implies that the Hurwitz number vanishes whenever $r_2$ does not have the same parity
of $\ell(\mu)+\ell(\nu)$.

The spin Hurwitz numbers  of the spin curve $(\P^1,\O(-1))$ also count ramified covers of $\P^1$, but with only odd ramifications and with sign induced from $\O(-1)$.

Specifically, suppose $\rho$ and $\s$ are odd partitions of $d$ (i.e., all parts in $\rho$ and $\s$ are odd) and consider possibly disconnected ramified covers $f:C\to \P^1$ with ramification profiles $\rho$ over $0\in\P^1$, $\s$ over $\infty\in\P^1$ and $\eta_3$ over  other $r$ fixed points of $\P^1$. The Riemann-Hurwitz formula in this case provides the Euler characteristic $\chi(C)$ of the domain as
\begin{equation}\label{RH}
\chi(C) = \ell(\rho)+\ell(\s)-2r.
\end{equation}
Since the ramification divisor $R_f$ of $f:C\to \P^1$ is  even, the twisted pull-back bundle
$$
N_f = f^*\O(-1)\otimes \O_C(\tfrac12 R_f)
$$
is a square root of the canonical bundle of $C$ (or a theta characteristic on $C$).
Given odd partitions $\rho,\s$ of $d$ and the number $r$, the spin Hurwitz number of $(\P^1,\O(-1))$ is defined to be
\begin{equation}\label{sHn}
H^{0,+}_d(\rho,\s,\eta_3^r) = \sum_f \frac{(-1)^{h^0(N_f)}}{|\Aut(f)|},
\end{equation}
where the superscript $+$ denotes the parity of the spin curve $(\P^1,\O(-1))$.

Let $p=(p_1,p_2,\cdots)$ and $p^\prime=(p_1^\prime,p_2^\prime,\cdots)$ be two sets of  variables where $p_n$ and $p_n^\prime$ are power-sum symmetric functions.
For a partition $\mu=(\mu_1,\mu_2,\cdots)$, let
$p_\mu=\prod p_{\mu_i}$. Let $\PP(d)$ be the set of partitions of $d$. Now
introduce a generating function of the Hurwitz numbers (\ref{Hn}) by
\begin{align*}
&\Phi(p,p^\prime,b,q)
\notag \\
=\
&1+\sum_{d>0} q^d \sum_{\mu,\nu\in\PP(d)} p_\mu p^\prime_\nu
\sum_{s=0}^\infty \frac{b^s}{s!} \sum_{\substack{r_i\geq 0, \sum r_i=s}}\frac{s!}{r_1! r_2! r_3!}
\Big( \frac{d^2+d}{2}\Big)^{r_1} H^0_d(\mu,\nu,\eta_2^{r_2},\eta_3^{r_3}).
\end{align*}
Under the restriction $p_2=p_4=\cdots=0$ and $ p^\prime_2=p^\prime_4=\cdots=0$,
the function $\Phi(p,p^\prime,b,q)$ specializes to the function $\Phi(p_\sb,p_\sb^\prime,b,q)$ where
$$
p_\sb=(p_1,0,p_3,0,\cdots)\ \ \ \ \ \text{and}\ \ \ \ \
p^\prime_\sb=(p^\prime_1,0,p_3^\prime,0,\cdots).
$$

Let $\OP(d)$ denote the set of odd partitions of $d$ and introduce a generating function of
the spin Hurwitz numbers (\ref{sHn}) by
\begin{align*}
&\Phi_B(p_\sb,p_\sb^\prime,b,q)
\\
=\ &1+
\sum_{d>0}  q^d \sum_{\rho,\s\in\OP(d)} p_\rho p^\prime_\s
\sum_{s=0}^\infty \frac{b^s}{s!} \sum_{r=0}^s \binom{s}{r} d^{2(s-r)}
2^{-\frac{\chi(C)}{2}} H^{0,+}_d(\rho,\s,\eta_3^r),
\end{align*}
where $\chi(C)$ is the domain Euler characteristic in (\ref{RH}).

\begin{theorem}\label{main}
The function $\Phi(p,p^\prime,b,q)$ is a tau function of the two-component KP (2-KP) hierarchy and
the function $\Phi_B(p_\sb,p_\sb^\prime,b,q)$ is a tau function of the two-component BKP (2-BKP) hierarchy such that
\begin{equation}\label{square}
\Phi^2_B(p_\sb,p_\sb^\prime,b,q) = \Phi(p_\sb,p_\sb^\prime,b,q).
\end{equation}
\end{theorem}

The proof of this theorem is based on the reduction of the KP hierarchy to the BKP hierarchy.
Both hierarchies are formulated by a single tau function such that
the square of the BKP tau function is a KP tau function
(cf. Proposition~4 of \cite{DJKM-B} and Proposition~1 of \cite{Y}).

In Section 2 we express the generating functions $\Phi$ and $\Phi_B$ via symmetric functions.
In Section 3 the famous Boson-Fermion correspondence converts those expressions to vacuum expectations of corresponding operators. We use the vacuum expectations to prove Theorem~\ref{main} in Section 4.
A discussion on a conjectural spin curve analog of the GW/H correspondence   is presented in Section~\ref{SA}.

\medskip
\non
{\bf Acknowledgement.} The author sincerely thank the referee for comments and suggestions that helped to improve the presentation of this paper.

\section{Symmetric functions}
\label{SF}

With transition matrices between linear bases of algebras relevant to our case, we express the Hurwitz generating function $\Phi$ via Schur functions and shifted symmetric power sums, whereas we express the spin Hurwitz generating function $\Phi_B$  via Schur Q-functions and odd power-sum symmetric functions.

\subsection{Hurwitz numbers}

\subsubsection{Central characters of the symmetric group}

Irreducible representations  and conjugacy classes of the symmetric group $\SS(d)$ on $d$ letters are indexed by partitions of $d$.  Let $C_\mu$ denote the conjugacy class indexed by $\mu\in\PP(d)$. The order of the centralizer of any element of $C_\mu$ is
\begin{equation}\label{oc}
z_\mu = \prod_k \mu(k)!k^{\mu(k)},
\end{equation}
where $\mu(k)$ is the number of parts in $\mu$ equal to $k$.

Let $\pi^\la$ be the irreducible representation indexed by $\la\in\PP(d)$ and $\chi^\la$ be its character.
The class sum $\sum_{g\in C_\mu} g$  acts on $\pi^\la$ as multiplication by constant. This constant is the central character
\begin{equation}\label{CC}
\f_\mu^\la = \frac{|C_\mu|}{\dim\, \pi^\la}\, \chi^\la_\mu,
\end{equation}
where $|C_\mu|=d!/z_\mu$ and $\chi^\la_\mu$ is the value of the character $\chi^\la$ on any element of $C_\mu$.
The character formula for the Hurwitz number (\ref{Hn}) is given as
\begin{equation}\label{CF1}
H^0_d(\mu,\nu,\eta_2^{r_2},\eta_3^{r_3}) = \sum_{\la\in\PP(d)}\,
\frac{\chi^\la_\mu}{z_\mu}\,\frac{\chi^\la_\nu}{z_\nu}\,
\big(\f_{\eta_2}^\la\big)^{r_2}\,\big(\f_{\eta_3}^\la\big)^{r_3}.
\end{equation}

\subsubsection{Shifted symmetric power sums}

The shifted symmetric power sum $\p_n$ of degree $n$ is defined by
$$
\p_n(\la) = \sum_{i=1}^\infty \Big((\la_i-i+\tfrac12)^n - (-i+\tfrac12)^n\Big).
$$
We denote by $\L^*$ the algebra generated by $\p_1,\p_2,\cdots$. From the central characters  of the symmetric group, Kerov and Olshanski \cite{KO} obtained  functions in $\L^*$ as follows. For any partition $\mu=(\mu_1,\cdots,\mu_\ell)$, let  $|\mu|=\sum \mu_i$ and set
$$
\mu\cup (1^k)=(\mu_1,\cdots,\mu_\ell,1,\cdots,1)\in \PP(|\mu|+k).
$$
Define a function $\f_\mu$ on the set of partitions by $\f_\mu(\la)=0$ if $|\la|<|\mu|$ and
\begin{equation*}\label{cha-fct}
\f_\mu(\la) = \binom{\mu(1)+k}{\mu(1)}\f_{\mu\cup (1^k)}^\la\ \ \
\text{if}\ \ \ k=|\la|-|\mu|\geq 0,
\end{equation*}
where $\mu(1)$ is the number of parts in $\mu$ equal to 1 as introduced in (\ref{oc}). Then
\begin{equation*}\label{tri-bases}
\f_\mu = \frac{1}{\prod \mu_i} \p_\mu + \cdots,
\end{equation*}
where $\p_\mu=\prod \p_{\mu_i}$ and the dots denote the lower degree terms (see \cite{KO} and also \cite{OP2}). This implies $\{\f_\mu\}$ and $\{\p_\mu\}$ are mutually triangular linear bases of $\L^*$.
Observe that for $\eta_3=(3,1^{d-3})$ in the character formula (\ref{CF1}),
$$
\f_{(3)}(\la)=\f_{\eta_3}^\la.
$$

The lemma  below is a consequence of
the Wassermann formula \cite{Wa} (see also \cite{IO}).
For a formal series $A(t)$, let
$$
[t^k]\{A(t)\} = \text{the coefficient of $t^k$ in $A(t)$.}
$$

\begin{lemma}\label{f3}
Let $\f_2=\f_{(2)}$ and $\f_3=\f_{(3)}$. We have:
\begin{itemize}
\item[(a)]
$\f_2 = \tfrac12 \p_2$.
\item[(b)]
$\f_3 =\tfrac13 \p_3 - \tfrac12 \p_1^2 + \tfrac{5}{12}\p_1$.
\end{itemize}
\end{lemma}

\begin{proof}
(a) is a well-known fact (cf. \cite{O}), which also follows by the same argument as for (b).
Consider the function $\p_3^\#$ on the set of partitions defined by
$$
\p_3^\#(\la) =
\left\{
\begin{array}{ll}
d^{\downarrow 3}\, {\ds \frac{\chi^\la_{(3,1^{d-3})}}{\dim \,\pi^\la} }, &d:=|\la|\geq 3
\\
0,&d<3,
\end{array}
\right.
$$
where $d^{\downarrow 3}=d(d-1)(d-2)$.
From (3.2) of \cite{IO}, one has
\begin{align*}
\p_3^\# &= [t^4]\Big\{-\tfrac13\prod_{j=1}^3\big(1-(j-\tfrac12)t\big)
\exp\Big(\sum_{j=1}^\infty \frac{\p_jt^j}{j}\big(1-(1-3t)^{-j}\big)\Big)\Big\}
\\
 &= \p_3 -\tfrac32 \p_1^2 +\tfrac54 \p_1.
\end{align*}
This together with (\ref{CC}) proves (b).
\end{proof}

\subsubsection{Schur functions}

The Schur functions $s_\la$ and the monomials
$p_\mu$ are related by
\begin{equation}\label{Schur}
s_\la(p) = \sum_{\mu\in \PP(d)} \frac{\chi^\la_\mu}{z_\mu}\,p_\mu,
\end{equation}
where $\la\in\PP(d)$ (see page 114 of \cite{Mac}).
Now by (\ref{CF1}) and (\ref{Schur}) one can write the generating function $\Phi(p,p^\prime,b,q)$ of Hurwitz numbers as
\begin{equation}\label{G1-S}
\Phi(p,p^\prime,b,q)  = 1+
\sum_{d> 0} \sum_{\la\in \PP(d)}q^d e^{b(\frac12(d+d^2)+\f_2(\la)+\f_3(\la))}
s_\la(p)s_\la(p^\prime).
\end{equation}

\subsection{Spin Hurwitz numbers}

\subsubsection{Central characters of the Sergeev group}
\label{ccsg}

The Sergeev group $\SC(d)$ is the semidirect product $\Cliff(d)\rtimes \SS(d)$, where $\Cliff(d)$
is the Clifford group generated by $\xi_1,\cdots,\xi_d$ and a central element $\ep$ subject to
the relations
$$
\xi_i^2 = 1,\ \ \ \ \ \ep^2 = 1,\ \ \ \ \ \xi_i\xi_j = \ep\xi_j\xi_i\ \ \ (i\ne j)
$$
and the symmetric group $\SS(d)$ acts on $\Cliff(d)$ by permuting the $\xi_i$'s.

Let $\SP(d)$ denote the set of strict partitions $\la=(\la_1>\la_2>\cdots)$ of $d$.
A spin $\SC(d)$-supermodule is a supermodule over the group algebra $\C[\SC(d)]$ on which
$\ep$ acts as multiplication by $-1$ such that
(i) the irreducible spin $\SC(d)$-supermodules $V^\la$ are indexed by strict partitions $\la\in\SP(d)$, (ii)
the character $\zeta^\la$ of $V^\la$ is determined by the values $\zeta^\la_\rho$ on the conjugacy classes $C_\rho$ indexed by odd partitions $\rho\in\OP(d)$, and (iii)
for the conjugacy class $C_\rho$ indexed by $\rho\in\OP(d)$,
$$
|C_\rho| = \frac{|\SC(d)|}{2^{\ell(\rho)+1}z_\rho} = \frac{2^{|\rho|-\ell(\rho)}|\rho|!}{z_\rho},
$$
where $z_\rho$ is introduced in (\ref{oc})
(see \cite{Jo, Se} and also \cite{I2}).
The class sum of $C_\rho$ acts on $V^\la$ as multiplication by a constant,
which is the central character of $V^\la$ given by
\begin{equation}\label{CC-2}
f_\rho^\la = \frac{|C_\rho|}{\dim V^\la}\,\zeta^\la_\rho.
\end{equation}

For $\la\in\SP(d)$, let
$$
\delta(\la) =
\left\{
\begin{array}{ll}
0,\  &\text{if $\ell(\la)$ is even},\\
1,\  &\text{if $\ell(\la)$ is odd}.
\end{array}
\right.
$$
From the Gunningham formula \cite{G,L}, one has the character formula for the spin Hurwitz number (\ref{sHn}):
\begin{equation}\label{CF2}
H^{0,+}_d(\rho,\s,\eta_3^r)
= 2^{\frac{-\ell(\rho)-\ell(\s)-2r}{2}}
\sum_{\la\in\SP(d)} 2^{-\delta(\la)}
\frac{\zeta^\la_\rho}{z_\rho}\frac{\zeta^\la_\s}{z_\s}\big(f_{\eta_3}^\la\big)^r.
\end{equation}

\subsubsection{Odd power-sum symmetric functions}

Denote by $\G$ the algebra generated by  $p_1,p_3,\cdots$.
In the same manner as for $\f_\mu$, one can obtain functions $f_\rho$ in $\G$
from the central characters of the Sergeev group.  For any odd partition $\rho$, let
$f_\rho(\la)=0$ if $|\rho|>|\la|$ and
\begin{equation*}\label{cha-fctB}
f_\rho(\la)=\binom{\rho(1)+k}{\rho(1)}f_{\rho\cup (1^k)}^\la\ \ \
\text{if}\ \ \ k=|\la|-|\rho|\geq 0,
\end{equation*}
where $\rho(1)$ is the number of parts in $\rho$ equal to 1 as in (\ref{oc}). Then
\begin{equation*}\label{tri-bases-B}
f_\rho = \frac{1}{z_\rho} p_\rho + \cdots,
\end{equation*}
where the dots denote the lower degree terms (see Section 6 of \cite{I2}).
It follows that $\{f_\rho\}$ and $\{p_\rho:\text{$\rho$ is odd}\}$ are mutually triangular linear bases of  $\G$.

The lemma below follows from Proposition~3.4 of \cite{I3}.

\begin{lemma}\label{f3-B}
Let $f_3=f_{(3)}$. We have
\begin{equation*}\label{f3-B}
f_3 = \tfrac13 p_3 - p_1^2 +\tfrac23 p_1.
\end{equation*}
\end{lemma}

\begin{proof}
Consider the function $p^\#_3$ on the set of strict partitions defined by
$$
p_3^\#(\la) = \left\{
\begin{array}{ll}
d^{\downarrow 3}\, \dfrac{X^\la_{(3,1^{d-3})}}{X^\la_{(1^d)}}, &d:=|\la|\geq 3,\\
0,&d<3,
\end{array}
\right.
$$
where $X^\la_\mu$ corresponds to $X^\la_\mu(-1)$ in \cite{Mac} such that
$$
X^\la_{(3,1^{d-3})} = 2^{\frac{\ell(\la)-\delta(\la)-2(d-2)}{2}}\zeta^\la_{(3,1^{d-3})},
\ \ \ \ \
X^\la_{(1^d)} = 2^{\frac{\ell(\la)-\delta(\la)-2d}{2}}\dim V^\la
$$
(see Proposition~3.3 of \cite{I2}).
By  (3.5) of \cite{I3}, one has:
\begin{align*}
p^\#_3 &=
[t^4]\Big\{ -\tfrac{1}{6}(1-\tfrac32 t)\prod_{j=1}^2(1-jt)
\exp\Big( \sum_{j=1}^\infty \frac{2p_{2j-1}t^{2j-1}}{2j-1}
\big(1-(1-3t)^{-(2j-1)}\big)\Big)\Big\}
\\
&=
p_3-3p_1^2+2p_1.
\end{align*}
This together with (\ref{CC-2}) shows the lemma.
\end{proof}

\subsubsection{Schur Q-functions}

The algebra $\G$ has another linear basis formed by Schur $Q$-functions $\{Q_\la:\text{$\la$ is strict}\}$.
For $\la\in\SP(d)$,
\begin{equation*}\label{Schur-Q}
Q_\la(p_\sb) = \sum_{\rho\in\OP(d)} 2^{\frac{\ell(\la)-\delta(\la)}{2}}
\frac{\zeta^\la_\rho}{z_\rho} p_\rho
\end{equation*}
(see Corollary~3 of \cite{Jo}). Thus,
for $\frac12p_\sb=(\frac12 p_1,0,\frac12p_3,\cdots)$,
\begin{equation}\label{SQ-half}
Q_\la(\tfrac12p_\sb) = \sum_{\rho\in\OP(d)} 2^{\frac{\ell(\la)-\delta(\la)-2\ell(\rho)}{2}}
\frac{\zeta^\la_\rho}{z_\rho} p_\rho.
\end{equation}
By (\ref{CF2}) and (\ref{SQ-half}), one can write the generating function $\Phi_B(p_\sb,p_\sb^\prime,b,q)$
of spin Hurwitz numbers as
\begin{equation}\label{G3-Q}
\Phi_B(p_\sb,p_\sb^\prime,b,q) = 1+
\sum_{d> 0}  \sum_{\la\in\SP(d)} q^{d} e^{b ( d^2+ f_3(\la))} 2^{-\ell(\la) }
Q_\la(\tfrac12 p_\sb)Q_\la(\tfrac12 p_\sb^\prime).
\end{equation}


\section{The operator formalism}

\label{OF}

The fermion operator formalism is a handy tool for handling various generating functions.
We employ the formalism to express the generating functions (\ref{G1-S}) and (\ref{G3-Q}) as  vacuum expectations. For more detailed discussions of the formalism, we refer to Chapter 14 of \cite{Kac}, Appendix of \cite{O1} and Section 2 of \cite{OP2}.

\subsection{Hurwitz numbers}

\subsubsection{The infinite wedge space}
\label{IW}

Let $V$ be a vector space with basis $\{\u{k}\}$ indexed by the half-integers $k\in \Z+\frac12$.
The infinite wedge space (or fermion Fock space) is the vector space
$$
\Lambda^{\frac\infty2}V
$$
with basis $\{v_S\}$ where $S=(s_1>s_2>s_3>\cdots)$ with $s_i-s_{i+1}=1$ for $i\gg 0$ and
\begin{equation*}\label{basis}
v_S = \u{s_1}\w \u{s_2}\w \u{s_3}\w\cdots.
\end{equation*}

Denote by $\L_0$ the subspace of $\Lambda^{\frac\infty2}V$
generated by vectors $v_\la$ indexed
by partitions $\la=(\la_1,\la_2,\cdots)$. These are defined by
\begin{equation*}\label{vip}
v_\la = \u{\la_1-\tfrac12}\w \u{\la_2-\tfrac32}\w\cdots\w\u{\la_i-i+\tfrac12}\w\cdots.
\end{equation*}
The vector indexed by the empty partition $\emptyset=(0,0,\cdots)$ is the vacuum vector
$$
v_\emptyset = \u{-\tfrac12}\w \u{-\tfrac32}\w \u{-\tfrac52}\w\cdots.
$$

Let $(\cdot\,,\cdot)$ be the inner product on $\Lambda^{\frac\infty2}V$
for which $\{v_S\}$ is an orthonormal basis.
The vacuum expectation of an operator $A$ on $\Lambda^{\frac\infty2}V$ is defined as
$$
\lb A\rb := (Av_\emptyset,v_\emptyset).
$$

\subsubsection{Charged fermions}

For each $k\in\Z+\frac12$, the operator $\psi_k$ on $\Lambda^{\frac\infty2}V$ is defined by
$$
\psi_k\,v = \u{k}\w v,
$$
and the operator $\psi^*_k$ is the adjoint of $\psi_k$.
These are charged fermions and satisfy the canonical anti-commutative relations:
\begin{equation}\label{CAC}
\psi_k\psi^*_\ell+\psi^*_\ell\psi_k = \delta_{k,\ell},\ \ \ \ \
\psi_k\psi_\ell+\psi_\ell\psi_k = \psi_k^*\psi_\ell^*+\psi_\ell^*\psi_k^* = 0.
\end{equation}

Infinite sums of the quadratics $\psi_j\psi^*_j$ make sense as operators on $\Lambda^{\frac\infty2}V$
if we write them in terms of normal ordering defined by
\begin{equation*}\label{no}
:\psi_k\psi_\ell^*: \overset{def}{=} \psi_k\psi_\ell^*-\lb\psi_k\psi_\ell^*\rb =
\left\{
\begin{array}{rl}
\psi_k\psi_\ell^*,\ \ &\ell>0,\\
-\psi_\ell^*\psi_k,\ \ &\ell<0.
\end{array}
\right.
\end{equation*}

\subsubsection{Operators related to shifted symmetric power sums}
\label{OES}

For any integer $n\geq 0$, define
\begin{equation*}\label{mo}
E_n = \sum_{k\in\Z+\frac12} k^n:\psi_k\psi^*_k:.
\end{equation*}
One of the salient features of the half-integer infinite wedge  lies in the relation between the operators $E_n$ and the shifted symmetric power sums $\p_n$:
\begin{equation}\label{KP-eigen}
E_n v_\la = \p_n(\la)v_\la
\end{equation}
(cf. Section~2 of \cite{OP2}).

The operators $E_1$ and $E_0$  are the energy and charge operator.
As $\p_1(\la)=|\la|$, $E_1 v_\la=|\la|v_\la$. On the other hand,
$E_0v_\la=0$ and hence  the subspace $\L_0$
 is the 0-eigenspace of $E_0$ (cf. Appendix of \cite{O1}).

To express the generating function (\ref{G1-S}) of Hurwitz numbers as a vacuum expectation, we consider the following operator
$$
F =  \tfrac13 E_3 +\tfrac12 E_2 +\tfrac{11}{12}E_1.
$$
By Lemma~\ref{f3} and (\ref{KP-eigen}), for every vector $v_\la$ with $|\la|=d$ one has
\begin{equation}\label{f3-op}
F v_\la = \Big(\tfrac12\big(d+d^2\big)+\f_2(\la)+\f_3(\la)\Big)v_\la.
\end{equation}

As shown in \cite{O}, it is convenient to use the operator $q^{E_1}=e^{E_1\ln{q}}$  because it can be written as
$$
q^{E_1} =\sum_{d\geq 0} P_d\,q^d,
$$
where $P_d$ is the orthogonal projection onto the $d$-eigenspace of $E_1$.

\subsubsection{The Boson-Fermion correspondence}

Introduce a set $t=(t_1,t_2,\cdots)$ of Miwa variables $t_n=p_n/n$ and let $\a^*(t)$ denote the adjoint of the operator
$$
\a(t) = \sum_{n>0} t_n \sum_{k\in\Z+\frac12} :\psi_{k-n}\psi^*_k:.
$$
Then from the remarkable Boson-Fermion correspondence (cf. $\S$14.10 of \cite{Kac}), one has
\begin{equation*}\label{BFC}
e^{\a^*(t)}v_\emptyset = \sum_\la s_\la(p)v_\la,
\end{equation*}
where the sum is over all partitions $\la$. This together with (\ref{G1-S}) and (\ref{f3-op}) gives
\begin{equation}\label{G1-VE}
\Phi(p,p^\prime,b,q) =
\left\lb e^{\a(t)} q^{E_1} e^{bF}e^{\a^*(t^\prime)}\right\rb,
\end{equation}
where $t^\prime=(t_1^\prime,t_2^\prime,\cdots)$ is another set of Miwa variables  $t^\prime_n=p_n^\prime/n$.

\subsection{Spin Hurwitz numbers}

\subsubsection{Neutral fermions}

There is an involution $w$ on the charged fermions defined by
$$
w(\psi_k) = (-1)^{k+\frac12}\psi^*_{-k-1},\ \ \ \ \
w(\psi_k^*) = (-1)^{k+\frac12}\psi_{-k-1}.
$$
The neutral  fermions are defined as $\pm 1$-eigenvectors of the involution.

\begin{defn}\label{neutral fermion}
For $m\in\Z$ and $i=\sqrt{-1}$, let
\begin{equation*}\label{neutral fermion}
\phi_m = \frac{1}{\sqrt{2}}\Big(\psi_{m-\frac12}+ (-1)^m\psi^*_{-m-\frac12}\Big),
\ \ \ \ \
\hphi_m = \frac{i}{\sqrt{2}}\Big(\psi_{m-\frac12}- (-1)^m\psi^*_{-m-\frac12}\Big).
\end{equation*}
\end{defn}

By the canonical anti-commutative relations (\ref{CAC}) for charged fermions, the neutral fermions also satisfy the canonical anti-commutative relations:
\begin{equation}\label{CACR}
\phi_n\phi_m+\phi_m\phi_n = \hphi_m\hphi_n+\hphi_n\hphi_m = (-1)^m\delta_{m,-n},
\ \ \ \phi_m\hphi_n+\hphi_n\phi_m = 0.
\end{equation}

With the relations (\ref{CAC}), the charged fermions $\psi_i$ and $\psi_j^*$ generate a Clifford algebra, denoted $Cl$.
The neutral fermions $\phi_m$ and $\hphi_m$ respectively generate isomorphic subalgebras. The isomorphism is given by the involution on the Clifford algebra $Cl$
induced from the map $\phi_m \leftrightarrow \hphi_m$, which we denote by
\begin{equation}\label{iso}
X\ \mapsto\  \hat{X}.
\end{equation}

Let $Cl_B$ be the subalgebra of $Cl$ generated by $\phi_m$'s. There is a decomposition
$$
Cl_B = Cl_B^0\oplus Cl_B^1,
$$
where $Cl_B^p$ is spanned by all products of the form
$\phi_{m_1}\cdots\phi_{m_s}$  with $s\equiv p$ (mod 2).

Recall that the subspace $\L_0$ defined in Section~\ref{IW} has an orthonormal basis  $\{v_\la\}$. Its neutral analog is the subspace
\begin{equation}\label{nfs}
\text{span}\{Xv_\emptyset:X\in Cl_B^0\}.
\end{equation}
As $\phi_m\vac=0$ for $m<0$, this subspace is also spanned by vectors
$v^B_\la$ indexed by  strict partitions $\la=(\la_1>\cdots>\la_\ell)$, which are defined by
$$
v^B_\la = \left\{
\begin{array}{cl}
\phi_{\la_1}\cdots\phi_{\la_\ell}v_\emptyset
&\text{if $\ell$ is even,}\\
\sqrt{2}\phi_{\la_1}\cdots\phi_{\la_\ell}\phi_0v_\emptyset
&\text{if $\ell$ is odd.}
\end{array}
\right.
$$
As $\phi_m^*=(-1)^m\phi_{-m}$ and $\phi_0^2=\frac12$,
$\{v^B_\la\}$ is an orthonormal basis of  the subspace (\ref{nfs}).

\subsubsection{Operators related to odd power-sum symmetric functions}
\label{OEFO}

For any odd integer $n\geq 1$, define a neutral analog of the operator $E_n$ by
\begin{equation*}\label{mo-B}
E^B_n = \sum_{m>0}(-1)^m m^n \phi_m\phi_{-m}.
\end{equation*}
From the canonical anti-commutative relations (\ref{CACR}), it is  easy to see  that
$$
(-1)^m\phi_m\phi_{-m}v^B_\la = \left\{
\begin{array}{cl}
v_\la^B\ &\text{if $m=\la_j$ for some $j$,}\\
0\ &\text{otherwise.}
\end{array}\right.
$$
This implies  $v^B_\la$ is an eigenvector of $E^B_n$ with eigenvalue $p_n(\la)$, that is,
\begin{equation}\label{ef-B}
E^B_n v^B_\la   =   p_n(\la)v^B_\la.
\end{equation}
As $p_1(\la)=|\la|$, the operator $E_1^B$ plays the same role as the energy operator $E_1$.

The following lemma together with (\ref{KP-eigen}) and (\ref{ef-B}) makes a connection between odd power-sum symmetric functions and shifted symmetric power sums via neutral and charged fermions.

\begin{lemma}\label{cn-H}
For any odd integer $n\geq 1$,
\begin{equation*}\label{E:cn-H}
E^B_n+\hE^B_n =
\sum_{i=0}^n \binom{n}{i}\frac{E_{n-i}}{2^{i}},
\end{equation*}
where the operator $\hE^B_n$ is  defined by the isomorphism (\ref{iso}) in an obvious way.
\end{lemma}

\begin{proof}
For $m>0$, let $k=m-\frac12$. Then by Definition~\ref{neutral fermion}, we have
\begin{align*}
(-1)^m\big(\phi_m\phi_{-m} +\hphi_m\hphi_{-m}\big)=
\psi_k\psi^*_k + \psi^*_{-k-1}\psi_{-k-1}.
\end{align*}
From this, we obtain
\begin{align*}
E^B_n+\hE^B_n
&=\sum_{k>0}\big(k+\tfrac12\big)^n \psi_k\psi_k^* +
\sum_{k>0}\big(k+\tfrac12\big)^n\psi^*_{-k-1}\psi_{-k-1}
\\
&= \sum_{k>0}\big(k+\tfrac12\big)^n \psi_k\psi_k^* +
\sum_{\ell<-1}\big(-\ell-\tfrac12)^n\psi_\ell^*\psi_\ell
\\
&=\sum_{k\in\Z+\frac12} \big(k+\tfrac12)^n:\psi_k\psi^*_k:
\ =\sum_{i=0}^n \binom{n}{i}\frac{E_{n-i}}{2^{i}},
\end{align*}
where the third equality follows because $n$ is odd and $-\psi^*_\ell\psi_\ell=\ :\psi_\ell\psi^*_\ell:$
for $\ell<0$.
\end{proof}

To express the generating function (\ref{G3-Q}) of spin Hurwitz numbers
as a vacuum expectation, consider the following operator
$$
F^B= \tfrac13 E^B_3 +\tfrac23 E^B_1.
$$
It follows from  Lemma~\ref{f3-B} that  for every vector $v^B_\la$ with $|\la|=d$,
\begin{equation}\label{fb3-op}
F^Bv^B_\la =  \big(d^2 +f_3(\la)\big)v^B_\la.
\end{equation}

\subsubsection{The neutral Boson-Fermion correspondence}

Let $\b^*(t)$ be the adjoint of the operator
\begin{align*}
\b(t) = \frac12 \sum_{n\geq 0}t_{2n+1}\sum_{m\in \Z} (-1)^{m+1}\phi_m\phi_{-m-2n-1}.
\end{align*}
Then from the neutral Boson-Fermion correspondence, one has
\begin{equation*}\label{BFC-B}
e^{\b^*(t)}\vac =
\sum_{\la} 2^{-\frac{\ell(\la) }{2}}Q_\la\big(\tfrac12 p_\sb\big)
v_\la^B,
\end{equation*}
where the sums are over all strict partitions $\la$ (see \cite{Y}).
This together with (\ref{G3-Q}) and (\ref{fb3-op}) yields:
\begin{equation}\label{G3-VE}
\Phi_B(p_\sb,p_\sb^\prime,b,q)
 =
\left\lb
e^{\b(t)}q^{E^B_1} e^{b   F^B}e^{\b^*(t^\prime)}
\right\rb
=
\left\lb
e^{\hb(t)}q^{\hE^B_1} e^{b   \hF^B}e^{\hb^*(t^\prime)}
\right\rb,
\end{equation}
where the second vacuum expectation is given by the isomorphism (\ref{iso}).

\section{Square root}
\label{SR}
Integrable hierarchies of KP type (including the 2-Toda lattice hierarchy) have fermionic forms of Hirota equations and fermionic formulas for tau functions (see \cite{AZ} and also Appendix of \cite{OST}). We apply those in the proof of Theorem~\ref{main} below.

\medskip
\non
{\bf Proof of Theorem~\ref{main}:}
The operators $q^{E_1}e^{bF}$ and $q^{E^B_1}e^{bF^B}$
satisfy the following fermionic forms of Hirota equations:
\begin{align}
\Big[q^{E_1}e^{bF}\otimes q^{E_1}e^{bF}&,
\sum_{k\in \Z+\frac12} \psi_k\otimes\psi^*_k\Big] = 0,
\label{HE1}
\\
\Big[q^{E^B_1}e^{bF^B}\otimes q^{E^B_1}e^{bF^B},
&\sum_{m\in \Z} (-1)^m\phi_m\otimes \phi_{-m}\Big] = 0.
\label{HE2}
\end{align}
One can obtain these commutation relations from: for $n\geq 0$,
$$
E_n\psi_k = \psi_k(E_n+k^n),\ \ \ \
E_n\psi^*_k = \psi_k^*(E_n-k^n),
$$
and for odd  $n\geq 1$,
$$
E^B_{n}\phi_{\pm m} = \phi_{\pm m}(E^B_{n}\pm m^{n}).
$$
It follows from (\ref{HE1}) that
the sequence
\begin{equation*}\label{TL-seq}
\tau_n(t,t^\prime) = \Big(e^{\a(t)}q^{E_1}e^{bF} e^{\a^*(t^\prime)}v_n,v_n\Big)\ \ \  (n\in\Z)
\end{equation*}
is a sequence of tau functions of the 2-Toda lattice hierarchy where
$v_n$ is the vacuum vector in the $n$-eigenspace of the charge operator $E_0$ defined by
$$
v_n = \u{n-\tfrac12}\w \u{n-\tfrac32}\w\u{n-\tfrac52}\w\cdots
$$
(see Appendix of \cite{O1}).
The Hurwitz generating function $\Phi(p,p^\prime,b,q)=\tau_0(t,t^\prime)$ by (\ref{G1-VE}) and
$\tau_0(t,t^\prime)$ is a tau function of the 2-KP hierarchy (see Theorem 1.12 of \cite{UT}).

It also follows from (\ref{HE2}) that the vacuum expectation in (\ref{G3-VE}),
$$
\Phi_B(p_\sb,p_\sb^\prime,b,q) =
\left\lb e^{\b(t)}q^{E^B_1} e^{b   F^B}e^{\b^*(t^\prime)}\right\rb,
$$
is a tau function of the 2-BKP hierarchy in time variables $t$ and $t^\prime$
 (see \cite{T}).

Now it remains to prove (\ref{square}).
By (\ref{G1-VE}) we have
\begin{equation*}\label{G2-VE}
\Phi(p_\sb,p_\sb^\prime,b,q) =
\left\lb e^{\a_B(t)} q^{E_1} e^{bF}e^{\a_B^*(t^\prime)}\right\rb,
\end{equation*}
where $\a_B(t)  =  \a(t_1,0,t_3,0,\cdots)$.  By definition, one has
$$
\b(t)+\hb(t)=\a_B(t).
$$
Since $\L_0$ is the kernel of the charge operator $E_0$,  Lemma~\ref{cn-H} shows that
$$
\big(E_1^B+\hE^B_1\big)|_{\L_0}=E_1,\ \ \
\big(F^B+ \hF^B\big)|_{\L_0}=F.
$$
Therefore noting $\lb Z\rb\lb \hat{W}\rb=\lb Z\hat{W}\rb$
for $Z,W\in Cl_B^0$ (see Lemma~1 of \cite{Y}),
we conclude
\begin{align*}
\Phi^2_B(p_\sb,p_\sb^\prime,b,q) &=
\left\lb
e^{\b(t)}q^{E^B_1} e^{bF^B}e^{\b^*(t^\prime)}
e^{\hb(t)}q^{\hE^B_1} e^{b\hF^B}e^{\hb^*(t^\prime)}
\right\rb
\\
&=
\left\lb
e^{\b(t)+\hb(t)}q^{E^B_1+\hE^B_1} e^{b   (F^B+\hF^B)}e^{\b^*(t^\prime)+\hb^*(t^\prime)}
\right\rb
 = \Phi(p_\sb,p_\sb^\prime,b,q),
\end{align*}
where the second equality follows since $Z\hat{W}=\hat{W}Z$ for $Z,W\in Cl_B^0$.
This completes the proof of Theorem~\ref{main}.

\section{Conjectural spin GW/H correspondence}
\label{SA}

In \cite{OP2}, the GW/H correspondence was defined by the two linear bases $\{\p_\mu\}$ and $\{\f_\mu\}$ of the algebra $\L^*$.
Let $\PP_\Q$ denote the vector space over $\Q$ with  basis the set of partitions (i.e., every vector in $\PP_\Q$ is a formal linear combination of partitions).
As $\{\f_\mu\}$ is a linear basis of $\L^*$, there is
a linear isomorphism
$$
\vp:\PP_\Q\to \L^*
$$
given by $\mu\mapsto \f_\mu$.
Using this isomorphism, one can extend the character formula  for Hurwitz numbers
of the curve of genus $h$ to the following multilinear form on $\PP_\Q$:
$$
H^h_d(v_1,\cdots,v_n)\ =\ \sum_{\la\in\PP(d)}\Big(\frac{\dim \,\pi^\la}{d!}\Big)^{2-2h}
\prod_{i=1}^n \vp(v_i)(\la).
$$
The GW/H correspondence is then the following equality:
$$
\Big\lb  \tau_{k_1}(\o)\cdots \tau_{k_n}(\o)  \Big\rb_d^{\bu\,h} =
H^h_d\left( \vp^{-1}\Big(\frac{\p^{\reg}_{k_1+1}}{(k_1+1)!}\Big),\cdots,
\vp^{-1}\Big(\frac{\p^{\reg}_{k_n+1}}{(k_n+1)!}\Big)\right).
$$
Here the left-hand side is the degree $d$ descendent GW invariants of the curve of genus $h$, $\o$ is the Poincar\'{e} dual of the point class, $\bullet$ denotes the disconnected theory and
$\p^{\reg}_n$ is the regularized shifted symmetric power sum defined by
$$
\p^{\reg}_n=\p_n+(1-2^{-n})\zeta(-n),
$$
where $\zeta$ is the Riemann zeta function \cite{OP2}.

One may ask whether interplays between the following functions yield connections between theories in (\ref{figure}):
\begin{gather*}
\begin{aligned}
\xymatrix{
\{\p_\mu\} \ar@{<->}[rr]^-{\text{GW/H}} \ar@{<->}[d]_{}
&& \{\f_\mu\}\ar@{<->}[d]^-{\text{Theorem~\ref{main}}}
\\
\{p_\rho\} \ar@{<->}[rr]_-{}
&& \{f_\rho\}
}
\end{aligned}
\end{gather*}

We will use the two linear bases $\{p_\rho:\text{$\rho$ is odd}\}$ and $\{f_\rho\}$  of the algebra  $\G$ to describe a conjectural spin curve analog of the GW/H correspondence.
Let $\OP_\Q$ denote the vector space over $\Q$ with basis the set of odd partitions.
There is a linear isomorphism $\vp_\sb:\OP_\Q\to \G$ given by
$$
\vp_\sb(\rho) = 2^{\frac{\ell(\rho)-|\rho|}{2}}f_\rho.
$$
Using this isomorphism, one can also extend the Gunningham formula \cite{G,L} for spin Hurwitz numbers of the spin curve of genus $h$ and parity $p$ to the following multilinear form on $\OP_\Q$:
\begin{align*}
&H^{h,p}_d(w_1,\cdots,w_n)
\\
=\
&2^{(d+1)(1-h)} \sum_{\la\in {\rm SP(d)} }
(-1)^{p\delta(\la)}
\Big( 2^{\frac{1-\delta(\la)}{2}}\, \frac{\dim V^\la}{|\SC(d)|} \Big)^{2-2h}
\prod_{i=1}^n \vp_\sb(w_i)(\la).
\end{align*}
Then the Maulik-Pandharipande formulae ((8) and (9) of \cite{MP}), which were proved in
\cite{KiL2,L0,KT}, show that for degree $d=1,2$,
the descendent GW invariants of the spin curve of genus $h$ and parity $p$ are given by
\begin{gather}\label{spinGH}
\begin{aligned}
&\Big\lb \tau_{k_1}(\o)\cdots\tau_{k_n}(\o)   \Big\rb_d^{\bu\,h,p}
\\
=\ &
H^{h,p}_d\left( \vp_\sb^{-1}\Big(\frac{ (-1)^{k_1}k_1!}{2^{k_1}(2k_1+1)!}\  p_{2k_1+1}\Big),
\cdots, \vp_\sb^{-1}\Big(\frac{ (-1)^{k_n}k_n!}{2^{k_n}(2k_n+1)!}\  p_{2k_n+1}\Big)\right).
\end{aligned}
\end{gather}

\begin{q}
Does the equality (\ref{spinGH}) hold for $d\geq 3$?
\end{q}

\noindent {\em  Department of  Mathematics,  University of Central Florida, Orlando, FL 32816}

\noindent {\em e-mail:}\ \ {\ttfamily junho.lee@ucf.edu}


\end{document}